\documentclass[a4paper,10pt]{article}

\usepackage{amssymb}
\usepackage{amsmath, enumerate,amsthm,fullpage,wasysym,textcomp}
\usepackage{epsfig}
\newtheorem{theorem}{Theorem}[section]
\newtheorem{mtheorem}{Theorem}
\newtheorem{mmtheorem}{Theorem}

\newtheorem{corollary}[theorem]{Corollary}

\newtheorem{definition}[theorem]{Definition}
\newtheorem{example}[theorem]{Example}

\newtheorem{lemma}[theorem]{Lemma}

\newtheorem{proposition}[theorem]{Proposition}
\newtheorem{remark}[theorem]{Remark}

\title{Statistics of blocks in $k$-divisible non-crossing partitions}

\author{Octavio Arizmendi\footnote{Supported by DFG-Deutsche Forschungsgemeinschaft Project SP419/8-1 E-mail: arizmendi@math.uni-sb.de} \\ Universit\"{a}t des Saarlandes, FR $6.1-$Mathematik,\\ 66123 Saarbr\"{u}cken, Germany}

\begin{document}

\date{\today}

\maketitle

\begin{abstract}
We derive a formula for the expected number of blocks of a given size from a non-crossing partition chosen uniformly at random. Moreover, we refine this result subject to the restriction of having a number of blocks given.
Furthermore, we generalize to $k$-divisible partitions. In particular, we find that, asymptotically, the expected number of blocks of size $t$ of a $k$-divisible non-crossing partition of $nk$ elements chosen uniformly at random is $\frac{kn+1}{(k+1)^{t+1}}$. Similar results are obtained for type $B$ and type $D$ k-divisible non-crossing partitions of Armstrong.
\end{abstract}


\section*{Introduction}

In this paper we study some statistics of the block structure of non-crossing partitions. A first systematic study of non-crossing partitions was done by G. Kreweras \cite{Kr72}. More recently, much more attention has been paid to non-crossing partitions because, among other reasons, they play a central role in the combinatorial approach of Speicher to Voiculescu's free probability theory \cite{VoDyNi92}. For an introduction to this combinatorial approach, see \cite{NiSp06}.

A non-crossing partition of $\{1,...,kn\}$ is called $k$-divisible if the size of each block is divisible by $k$. The poset of $k$-divisible non-crossing partitions was introduced by Edelman \cite{Edel1} and reduces to the poset of all non-crossing partitions for $k=1$. We denote these posets by $NC^k(n)$ and $NC(n)$, respectively.

As can be seen in \cite{Ar1} and \cite{ArVar}, $k$-divisible non-crossing partitions play an important role in the calculation of the free cumulants and moments of products of $k$ free random variables. Moreover,  in the approach given in \cite{ArVar} for studying asymptotic behavior of the size of the support, when $k\rightarrow\infty$, understanding the asymptotic behavior of the sizes of blocks was a crucial step. 

In this direction, a recent paper by Ortmann \cite{Ort} studies the asymptotic behavior of the sizes of the blocks of a uniformly chosen random partition. This lead him to a formula for the right-edge of the support of a measure in terms of the free cumulants, when these are positive. He noticed a very simple picture of this statistic as $n\rightarrow\infty$. Roughly speaking, in average, out of the $\frac{n+1}{2}$ blocks of this random partition, half of them are singletons, one fourth of the blocks are pairings, one eighth of the blocks have size $3$, and so on.

Trying to get a better understanding of this asymptotic behavior,
the question of the exact calculation of this statistic arose. In this paper, we answer this question and refine these results by considering the number of blocks given. Moreover, we generalize to $k$-divisible partitions, as follows.

\begin{mtheorem}\label{T1}
The sum of the number of blocks of size $tk$ over all the $k$-divisible non-crossing partitions of $\{1,2,..,kn\}$ is given by
\begin{equation}
\binom{n(k+1)-t-1}{nk-1}.
\end{equation}
\end{mtheorem} 

In particular, asymptotically, we have a similar phenomena as for the case $k=1$;  about a $\frac{k}{k+1}$ portion of all the blocks have size $k$, then a $\frac{k}{(k+1)^2}$ portion have size $2k$, then $\frac{k}{(k+1)^3}$ are of size $3k$, etc.

More generally, for any Coxeter Group $W$, Bessis \cite{Bes} and Brady and Watt \cite{BW} defined the poset $NC(W)$ of non-crossing partitions for a reflexion group, making $NC(A_{n-1})$ isomorphic to $NC(n)$. Furthermore, Armstrong \cite{Arm2009} defined the poset $NC^k(W)$ of $k$-divisible non-crossing partitions for a reflexion group. Many enumerative results are now known for this noncrossing partitions for reflection groups
(see e.g. \cite{ReiAs,Kra,KraMu,Kim,Rei}).

 However, analogous results to Theorem \ref{T1} have not been studied. We address this problem for the types $B$ and $D$. 
\begin{mtheorem}\label{T2}\begin{enumerate}[{\rm (1)}]
\item The sum of the number of non-zero pairs of blocks $\{-V,V\}$ of size $tk$ over all the non-crossing partitions in $NC^k_B(n)$ is given by
\begin{equation*}
nk\binom{n(k+1)-t-1}{nk-1}.
\end{equation*}
\item The number of  the non-crossing partitions in $NC(B_n)$ with a zero-block of size $2t$ is 
\begin{equation*}
\binom{n(k+1)-t-1}{nk-1}.
\end{equation*} 
\end{enumerate}
\end{mtheorem}
While for $k$-divisible partitions type $D$ we get the following.
\begin{mtheorem}\label{T3}
The sum of the number of pairs of blocks $\{-V,V\}$ of size $tk$ over all the non-crossing partitions in $NC^k_D(n)$ is given by
$$(k(n-1)+1)
\binom{(k+1)(n-1)-t}{k(n-1)-1}+k(n-1)\binom{(k+1)(n-1)-t-1}{k(n-1)-2}.$$
\end{mtheorem}

Theorems \ref{T2} and \ref{T3} show that, asymptotically, the behavior of  sizes is the same for type $B$ and type $D$ non-crossing partitions as for the classical non-crossing partitions (or type $A$). Again, a $\frac{k}{k+1}$ portion of all the blocks have size $k$, then a $\frac{k}{k+1}$ portion of the remaining blocks are of size $2k$,, etc.

Another consequence is that the expected number of blocks of a $k$-divisible non-crossing partition is given by $\frac{kn+1}{k+1}$. An equivalent formulation of this result was also observed by Armstrong \cite[Theorem 3.9]{Arm2009} for any Coxeter group. It is then a natural question if this simple formula can be derived in a bijective way.  We end with a bijective proof of this fact, for type $A$ and $B$ $k$-divisible non-crossing partitions.

Let us finally mention that there exists a type $B$ free probability. Free probability of type $B$ was introduced by Biane, Goodman and Nica \cite{BGN} and was later developed by Belinschi and Shlyakhtenko \cite{BelShly09}, Nica and F\'evrier \cite{FreNi} and Popa \cite{Po}.

Apart from this introduction, the paper is organized as follows. In Section \ref{S1} we give basic definitions and collect known results on the enumeration of classical non-crossing partitions and non-crossing partitions of type $B$ and $D$ . The main new enumerative results concerning the numbers of blocks of a given size for a non-crossing partition of type $A$, $B$ and $D$  chosen at random are given in Sections \ref{S2}, \ref{S3} and \ref{S4}, respectively. Finally, in Section \ref{S5} we give a bijective proof of the fact that in average the number of blocks is given by $\frac{kn+1}{k+1}$ and some further consequences of this bijection.

\section{Preliminaries}\label{S1}

\subsection*{Classical non-crossing partitions}

\begin{definition}\begin{enumerate}[{\rm (1)}]
\item  We call $\pi =\{V_{1},...,V_{r}\}$ a \textbf{partition} of the set $[n]:=\{1, 2,.., n\}$
if and only if $V_{i}$ $(1\leq i\leq r)$ are pairwise disjoint, non-void
subsets of $S$, such that $V_{1}\cup V_{2}...\cup V_{r}=\{1, 2,.., n\}$. We call $
V_{1},V_{2},...,V_{r}$ the \textbf{blocks} of $\pi $. The number of blocks of 
$\pi $ is denoted by $\left\vert \pi \right\vert $.

\item  A partition $\pi =\{V_{1},...,V_{r}\}$ is called \textbf{non-crossing} if for all $1 \leq a < b < c < d \leq n$
if $a,c\in V_{i}$ then there is no other subset $V_{j}$ with $j\neq i$ containing $b$ and $d$.

\item We say that a partition $\pi$  is \textbf{$k$-divisible} if the size of all the blocks is a multiple of $k$.  If all the blocks are exactly of size $k$ we say that $\pi$ is \textbf{$k$-equal}.
\end{enumerate}
\end{definition}

We  denote the set of non-crossing partitions of $[n]$ by $NC(n)$, the set of $k$-divisible non-crossing partitions of $[kn]$ by $NC^k(n)$ and the set of $k$-equal  non-crossing partitions of $[kn]$ by $NC_k(n)$. $NC^k(n)$ is a poset (see Remark \ref{kreweras} below) and was first studied by Edelman\cite{Edel1}, who calculated many of its enumerative invariants.

\begin{remark}\label{kreweras}
\begin{enumerate}[{\rm (1)}]
 \item $NC^k(n)$ can be equipped with the partial order $\preceq$ of reverse refinement ($\pi\preceq\sigma$ if and only if every block of $\pi$ is completely contained in a block of $\sigma$). 
\item For a given $\pi\in NC(n)\cong NC(\{1,3,\dots ,2n-1\})$ we define its \textit{Kreweras complement} $$Kr(\pi):=\max\{\sigma \in NC(2,4,\dots 2n):\pi\cup\sigma\in NC(2n)\}.$$ The map $Kr:NC(n)\to NC(n)$ is an order reversing isomorphism. Furthermore, for all $\pi \in NC(n)$ we have that $|\pi|+|Kr(\pi)|=n+1$, see \cite{NiSp06} for details.
\end{enumerate}
\end {remark}

There is a graphical representation of a partition $\pi\in NC(n) $ which makes clear the property of being crossing or non-crossing, usually called the circular representation. We think of $[n]$ as labelling the vertices of a regular $n$-gon, clockwise.
If we identify each block of $\pi $ with the convex hull of its
corresponding vertices, then we see that $\pi $ is non-crossing precisely
when its blocks are pairwise disjoint (that is, they don't cross). 

\begin{figure}[here]
\begin{center}
\epsfig{file=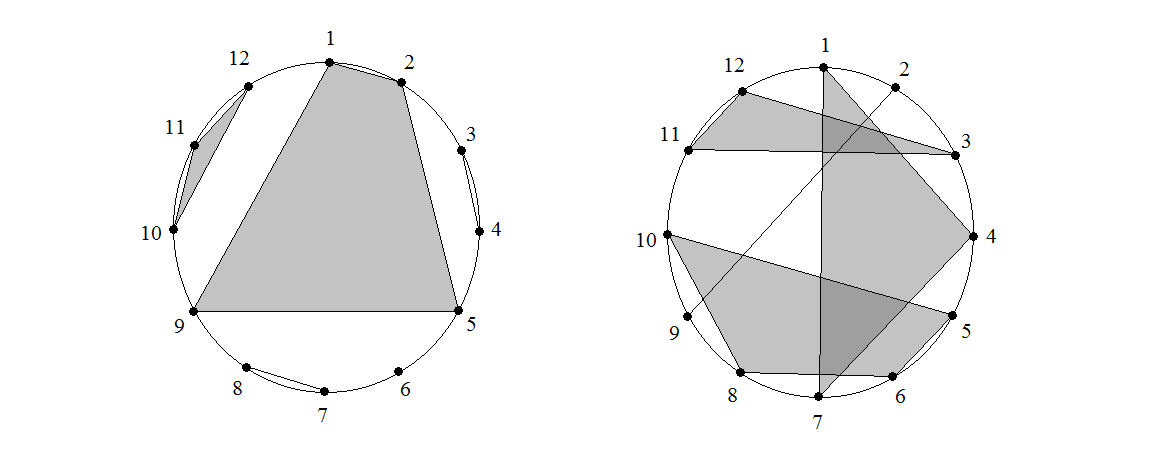, width=12cm}
\caption{Crossing and Non-Crossing Partitions}
\end{center}
\end{figure}

 Fig. 1 shows  the non-crossing partition $\{\{1,2,5,9\},\{3,4\},\{6\},\{7,8\},\{10,11,12\}\}$ of the set $[12]$, and the crossing
partition $\{ \{ 1,4,7\},\{ 2,9\}, \{ 3,11,12\} ,\{ 5,6,8,10\} \} $ of $[12]$  in their circular
representation.
\begin{remark} \label{4}The following characterization of non-crossing partitions is sometimes useful: for any $\pi\in NC(n)$, one can always find a block $V=\{r+1,\dots,r+s\}$ containing consecutive numbers. If one removes this block from $\pi$, the partition $\pi\setminus V\in NC(n-s)$ remains non-crossing.
\end{remark}

We recall the following result which gives a formula for the number of partitions with a given type \cite{Kr72}.
\begin{proposition}
\label{NC type}Let $r_{1},r_{2},...r_{n}$ be nonnegative integers such that $r_{1}+2r_{2}...+nr_{n}=n$. Then the number of
partitions of $\pi $ in $NC(n)$ with $r_{1}$ blocks of size $1$, $r_{2}$
blocks of size $2$ ,$\dots$, $r_{n}$ blocks of size $n$ equals
\begin{equation}
\frac{n!}{p_r(n-m+1)!}
\text{,} 
\label{rtype partitions}
\end{equation}
where $p_r=r_1!r_2!\cdots r_n!$ and $r_{1}+r_{2}...+r_{n}=m.$
\end{proposition}

It is well known that the number of non-crossing partition is given by the Catalan numbers $\frac{1}{n+1}\binom{2n}{n}$. More generally, for $k$-divisible non-crossing partitions we have the following \cite{Edel1}.

\begin{proposition}
\label{k-divisible}Let $NC^{k}(n)$ be the set of non-crossing partitions
of $[nk]$ whose sizes of blocks are multiples of $k$. Then 
\begin{equation*}
\#NC^{k}(n)=\frac{\binom{(k+1)n}{n}}{kn+1}.
\end{equation*}
\end{proposition}

On the other hand, from Proposition \ref{NC type}, we can easily count $k$-equal partitions. 
\begin{corollary}
\label{PartexactK}Let $NC_{k}(n)$ be the set of non-crossing partitions
of $nk$ whose blocks are of size of $k$. Then 
\begin{equation*}
\#NC_{k}(n)=\frac{\binom{kn}{n}}{(k-1)n+1}.
\end{equation*}
\end{corollary}

The reader may have noticed from Proposition \ref{k-divisible} and Corollary \ref{PartexactK} that the number of $(k+1)$-equal non-crossing partitions of $[n(k+1)]$ and the number of $k$-divisible non-crossing partitions of $[nk]$ coincide. We derive a bijective proof of this fact and study further consequences in Section $5$.

\subsection*{Non-crossing partitions for Coxeter groups of classical types}

For any Coxeter Group $W$, Bessis \cite{Bes} and Brady and Watt \cite{BW} defined the poset of non-crossing partition for the reflexion group $NC(W)$. 

For the three infinite families of finite irreducible Coxeter groups, $A_{n-1}$ (the symmetric group), $B_n$(the hyperoctahedral group) and $D_n$(an index $2$ subgroup of $B_n$) known as classical groups, there exists combinatorial realizations of $NC(W)$.
While $NC(A_{n-1})$ is isomorphic to $NC(n)$, the combinatorial realizations of the lattices $NC(B_n)$ and $NC(D_n)$, that we  explain in this section, were done by Reiner \cite{Rei} and Reiner and Athanasiadis\cite{ReiAs}, respectively. 
 
Furthermore, Armstrong \cite{Arm2009} defined the poset $NC^k(W)$ of $k$-divisible non-crossing partitions for the reflexion group $W$, making $NC^k(A_{n-1})$ isomorphic to $NC^k(n)$. The construction of Reiner can be easily be generalized to $NC^k(B_n)$. However, a combinatorial realization of $NC^k(D_{n})$ was not known, until the recent paper by Krattenthaler and M\"uller \cite{KraMu}.

We use these combinatorial realizations to define non-crossing partitions for the Coxeter groups $B_n$ and $D_n$.

Let us start with the definition of type $B_n$ partitions.

\begin{definition}
Let $[\pm n]:=\{1,2,...,n,-1,-2,...,-n\}$. A non-crossing partition of type $B_n$ is a non-crossing partition $\pi$ of $[\pm n]$ such that for each block $V$ of $\pi$ the block $-V=\{-x|x\in V\}$ is also a block of $\pi$, and there is at most one block, called the \textbf{zero-block}, which satisfies $V=-V$. We denote by $NC_B(n)$ the set of non-crossing partitions of type $B_n$.
\end{definition}

There is a circular representation for non-crossing partitions of type $B_n$ obtained as follows. Arrange $2n$ vertices on a circle and label clockwise with the integers $1, 2, . . . ,n,-1,-2, . . . ,-n$. For each block $V$ of $\pi$, draw the convex hull of the vertices whose labels are the integers in $B$. Then $\pi$ is a non-crossing partition of type $B_n$ if the convex hulls do not intersect. Notice that the circular representation of the non-crossing partitions of type $B_n$ correspond to the circular representation partitions in $NC(2n)$ which are centrally symmetric.

\begin{figure}[h]
\begin{center}
\epsfig{file=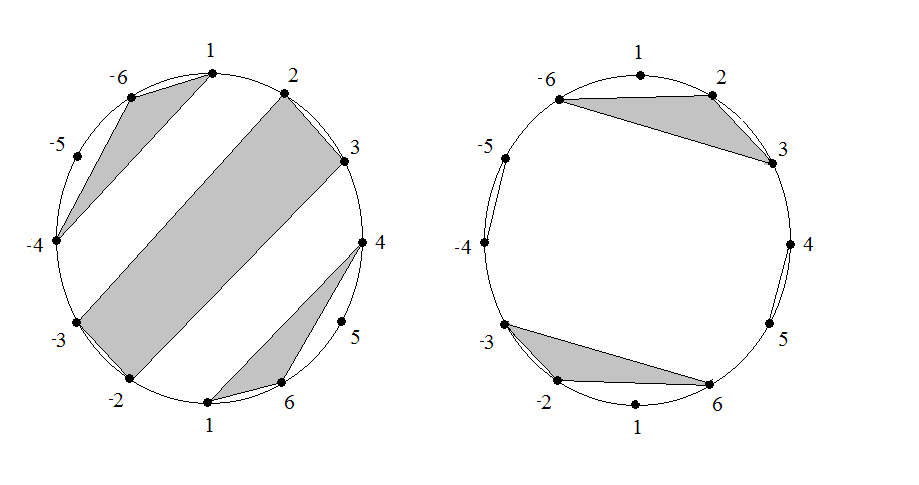, width=10cm}
\caption{Type $B$ partitions}
\end{center}
\end{figure}

The formula for the number of partitions of type $B_n$ with the sizes of blocks given was obtained by Athanasiadis \cite{As}.

\begin{proposition}
\label{NC type B} Let $s+r_{1},r_{2},...r_{n}$ be nonnegative integers such that $s+r_{1}+2r_{2}...+nr_{n}=n$. Then the number of
partitions of $\pi $ in $NC_B(n)$ with $r_{i}$ non-zero pairs of blocks of size $i$ and with a zero-block of size $2s$, is given by
\begin{equation}
\frac{n!}{p_r(n-m)!}  \label{rtype partitionsB}
\end{equation}
where $p_r=r_1!r_2!\cdots r_n!$ and $r_{1}+r_{2}...+r_{n}=m$.
\end{proposition}

We  define Type $D$ partitions via their circular representation.
The type $D$ circular representation of a partition $\pi\in NC_B(n)$ is the drawing obtained as follows. Arrange $2n-2$ vertices labeled with $1,2,...,n-1,-1,-2,...,-(n-1)$ on a circle and put a vertex labeled
with $\pm n$ at the center. For each block $V$ of $\pi$, draw the convex hull of the vertices whose labels are
in $B$. 

\begin{definition} A partition $\pi\in NC_B(n)$ is a non-crossing partition of type $ D_n$ if the convex hulls of its type $D_n$ circular representation do not intersect in their interiors and if there is a zero-block $V$ then $\{n,-n\}\in V$. We denote by $NC_D(n)$ the set of non-crossing partitions of type $D_n$. 
\end{definition}

The formula for the number of partitions of type $D$ with the sizes of the blocks given was obtained by Athanasiadis and Reiner \cite{ReiAs}.
\begin{proposition}
\label{NC type D} Let $s+r_{1},r_{2},...r_{n}$ be nonnegative integers such that $s+r_{1}+2r_{2}...+nr_{n}=n$. Then the number of
partitions of $\pi $ in $NC_D(n)$ with $r_{i}$ non-zero pairs of blocks of size $i$ and with a zero-block of size $2s$, is given by
\begin{equation}
\frac{(n-1)!}{p_r(n-m-1)!}~~~~~~~~~~~~~~~~~~~~~~~~~~\text{if} ~s\geq2
\end{equation}
\begin{equation}
2\frac{(n-1))!}{p_r((n-1)-m)!}+r_1\frac{(n-1)!}{p_r(n-m)!}~~~~~ \text{if} ~s=0 
\end{equation}
where $p_r=r_1!r_2!\cdots r_n!$ and $r_{1}+r_{2}...+r_{n}=m.$
\end{proposition}

\begin{figure}[h]
\begin{center}
\epsfig{file=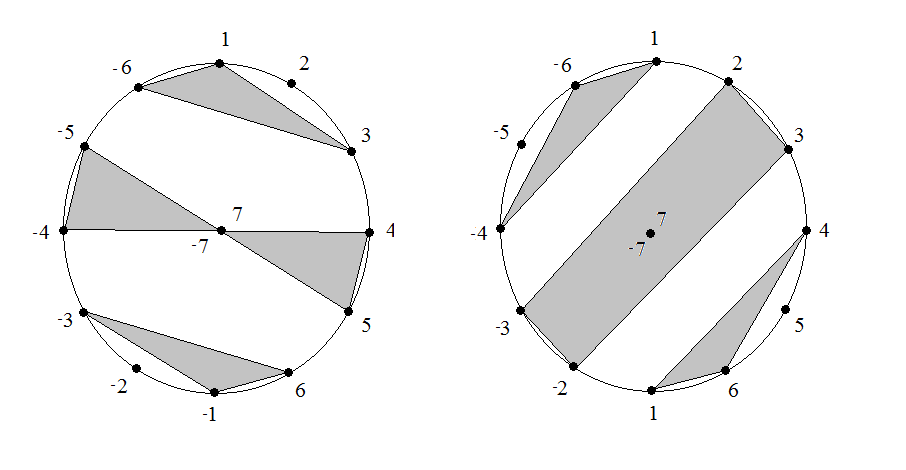, width=10cm}
\caption{Type $D$ partitions}
\end{center}
\end{figure}
\subsection*{Some basic combinatorial identities}
 For the convenience of the reader, we end this section by recalling some combinatorial identities that we use often in Sections \ref{S2}, \ref{S3} and \ref{S4}.

The following two summation lemmas will enable us to use Propositions \ref{NC type}, \ref{NC type B} and \ref{NC type D} to get the number of blocks of size $t$ subject to the restriction of having a fixed number $m$ of blocks.

\begin{lemma}\label{lemma0}
The following identity holds
\begin{equation}
\sum_{\substack{ r_{1}+r_{2}+\cdots r_{n}=m  \\ r_{1}+2r_{2}+\cdots (n)r_{n}=n}}\frac{m!%
}{r_1!\cdots r_n!}=\binom{n-1}{m-1}.
\end{equation}
\end{lemma}
\begin{proof}
This is proved easily by counting in two ways the number of paths from $(0,0)$ to $(n-1,m-1)$ using
the steps $(a,b)\rightarrow (a,b+1)$ or $(a,b)\rightarrow (a+1,b)$ by observing that $$\frac{(m+1)!}{r_1!\cdots r_n!}=\binom{m+1}{r_1}\binom{m+1-r_1}{r_2}\cdots\binom{m+1-(r_1+\cdots r_{n-1})}{r_n}.$$
\end{proof}

\begin{lemma}\label{lemma1} The following identity holds
 \begin{equation}\label{lemma1sum}
\sum_{\substack{ r_{1}+r_{2}+\cdots r_{n}=m  \\ %
r_{1}+2r_{2}+\cdots nr_{n}=n}}\frac{(m-1)!r_t}{r_1!\cdots r_n!}=\binom{n-t-1}{m-2}.  
 \end{equation}

\end{lemma}
\begin{proof}
 We make the change of variable $\tilde r_t=r_t-1$ and $\tilde r_i=r_i$ for $i\neq t$. Then
\begin{eqnarray*}\sum_{\substack{ r_{1}+r_{2}+\cdots r_{n}=m  \\ %
r_{1}+2r_{2}+\cdots nr_{n}=n}}\frac{(m-1)!r_t}{r_1!\cdots r_n!} 
&=& \sum_{\substack{ \tilde r_{1}+\tilde r_{2}+\cdots \tilde r_{n}=m-1  \\ %
\tilde r_{1}+2\tilde r_{2}+\cdots (n-t)\tilde r_{n-t}=n-t}}\frac{(m-1)!}{\tilde r_1!\cdots \tilde r_n!} \\
&=&\binom{n-t-1}{m-2},
\end{eqnarray*}
where we used the Lemma \ref{lemma0} in the last equality.

\end{proof}

Finally, we remind the so-called Chu-Vandermonde's identity which will enable us to remove the restriction of having a number of blocks given.

\begin{equation}\label{CHU}
\sum_{m=0}^{s}\binom{y}{m}\binom{x}{s-m}=\binom{x+y}{s}.
\end{equation}

\section{Number of blocks in k-divisible non-crossing partitions.}\label{S2}

First we calculate the expected number of blocks of a given size $t$, subject to the restriction of having $m$ blocks from which the main result will follow.

\begin{proposition} \label{3k}
The sum of the number of blocks of size $tk$ of all non-crossing partitions in $NC^k(n)$ with $m$ blocks
\begin{equation}
\binom{nk}{m-1}\binom{n-t-1}{m-2}. \label{pertype2}
\end{equation}
\end{proposition} 

\begin{proof}
First we treat the case $k=1$. In order to count the number of blocks of size $t$ of a given partition all partitions $\pi$ with $r_{1}$ blocks with size $1$, $r_{2}$
blocks of size $2$, $\dots$ , $r_{n}$ blocks of size $n$, we need to multiply by $r_t$. So we want to calculate the following sum
\begin{eqnarray*}
\sum_{_{\substack{ r_{1}+r_{2}+\cdots r_{n}=m  \\ r_{1}+2r_{2}+\cdots nr_{n}=n}}}
\frac{n!r_t}{(n+1-m)!p_r} &=&\binom{n}{m-1}\sum_{_{\substack{ r_{1}+r_{2}+\cdots r_{n}=m  \\ r_{1}+2r_{2}+\cdots nr_{n}=n}}} \frac{(m-1)!r_t}{p_r} \\
&=&\binom{n}{m-1}\binom{n-t-1}{m-2}.
\end{eqnarray*}%
We used Lemma \ref{lemma1} in the last equality. This solves the case $k=1$.

For the general case we follow the same strategy. In this case we need $(r_{1},\dots, r_{n})$ such that $r_{i}=0$ if $k$ does not divide ${i}$. So
the condition $r_{1}+r_{2}+\cdots +r_{n}=m$ is really $
r_{k}+r_{2k}+\cdots +r_{nk}=m $ and the condition $
r_{1}+2r_{2}+\cdots +nr_{n}=nk$ is really $kr_{k}+2kr_{2k}+\cdots +(nk)r_{nk}=nk,$
or equivalently $r_{k}+2r_{2k}+\cdots +nr_{nk}=n.$ Making the change of variable 
$r_{ik}=s_{i}$ we get.
\begin{equation}
\sum\limits_{\substack{r_{k}+r_{2k}+\cdots +r_{nk}=m  \\ %
kr_{k}+2kr_{2k}+\cdots +(nk)r_{nk}=nk}}\frac{(nk)!r_{tk}}{(nk+1-m)!
r_{k}!r_{2k}!\dots r_{nk}!}
=\sum\limits 
_{\substack{ s_{1}+s_{2}+\cdots +s_{n}=m  \\ s_{1}+2s_{2}+\cdots +ns_{n}=n}}\frac{(nk)!s_t}{(nk+1-m)!\prod\limits_{i=0}^{n}s_{i}!}. 
\end{equation}
Now, the last sum can be treated exactly as for the case $k=1$, yielding the result. This reduction to the case $k=1$ will be obviated for types $B$ and $D$.
\end{proof}

Now we can prove Theorem 1.

\begin{mmtheorem}
The sum of the number of blocks of size $tk$ over  $all$ the $k$-divisible non-crossing partitions of $\{1,...,kn\}$ is given by 
\begin{equation} \label{eq2}\binom{n(k+1)-t-1}{nk-1}.
\end{equation}
\end{mmtheorem} 

\begin{proof}
We use Proposition \ref{3k} and sum over all possible number of blocks. Letting $\tilde m=m-1$, we get $$\sum^{nk}_{m=1}\binom{nk}{m-1}\binom{n-t-1}{m-2}=\sum^{nk-1}_{\tilde m=0}\binom{nk}{nk-\tilde m}\binom{n-t-1}{\tilde m-1}.$$
Now, using the Chu-Vandermonde's identity
for $s=nk-1$, $x=n(k+1)-t-1$ and $y=nk-1$ we obtain the result.
\end{proof}

\begin{corollary}
The expected number of blocks of size $tk$ of a non-crossing partition chosen uniformly at random in $NC^k(n)$ is given by
\begin{equation} \label{expectedtypeA} 
\frac{(nk+1)\binom{n(k+1)-t-1}{nk-1}}{\binom{(k+1)n}{n}}.
\end{equation}
\end{corollary}

Moreover, similar to the case $k=1$, asymptotically the picture is very simple, about a $\frac{k}{k+1}$ portion of all the blocks have size $k$, then $\frac{k}{k+1}$ of the remaining blocks are of size $2k$, and so on. This is easily seen using  (\ref{expectedtypeA}).

\begin{corollary}
When $n\rightarrow\infty$ the expected number of blocks of size $tk$ of a non-crossing partition chosen uniformly at random in $NC^k(n)$ is asymptotically $\frac{nk}{(k+1)^{t+1}}$.
\end{corollary}

The following is a direct consequence of the last proposition. 

\begin{corollary}
The sum of the number of blocks of all the $k$-divisible non-crossing partitions in $NC^k(n)$ is $$\binom{n(k+1)-1}{nk}.$$
\end{corollary}
\begin{proof}
Summing over $t$, in (\ref{eq2}), we easily get the result.
\end{proof}

Finally, from the previous corollary one can calculate the expected number of block of $k$-divisible non-crossing partition.

\begin{corollary}\label{expected}
The expected number of blocks of a $k$-divisible partition of $[kn]$ chosen uniformly at random is given by $\frac{kn+1}{k+1}.$
\end{corollary}

\begin{remark}
1) Corollary \ref{expected} was proved by Armstrong \cite{Arm2009} for any Coxeter group.
 
2) When $k=1$ there is a nice proof of  Corollary \ref{expected}. Recall from Remark \ref{kreweras} that the Kreweras complement $Kr:NC(n)\rightarrow NC(n)$ is a bijection such that $|Kr(\pi)|+|\pi|=n+1$ . Then summing over all $\pi\in NC(n)$ and dividing by $2$ we obtain that the expected value is just $\frac{n+1}{2}$. This suggests that there should be a bijective proof of Corollary \ref{expected}. This is  done in Section \ref{S5}.
\end{remark}

\section{Number of blocks in $k$-divisible partitions of type B}\label{S3}

In this section we give analogous results for $k$-divisible non-crossing partitions of type $B$ defined in \cite{Arm2009}

\begin{definition}
A partition $\pi$ in $NC_B(n)$ is \textbf{$k$-divisible} if the size of all blocks in $\pi$ are multiples of $k$. We denote by $NC^k_B(n)$ the set of $k$-divisible non-crossing partitions of type $B_{kn}$.
\end{definition}

The number of $k$-divisible non-crossing partitions of type $B$ was given in \cite{ReiAs}.

\begin{proposition}
\label{k-divisibleB}Let $NC^k_B(n)$ be the set of non-crossing partitions
of $k$-divisible non-crossing partitions of type $B_{kn}$ Then 
\begin{equation*}
\#NC^k_B(n)=\binom{(k+1)n}{n}.
\end{equation*}
\end{proposition}
\begin{remark}The definition of a $k$-divisible partition of type $B_n$ seems quite natural. However, there is small detail on what the size of the zero-block shall be. Notice in the previous definition that for $k$ even and $n$ an integer, in order for a partition to be in $NC^k_B(n)$, the size of the zero-block is a multiple of $4$. Moreover, there exist $2$-divisible non-crossing partitions of type $B$ for $n$ odd, and for instance it makes sense to talk about  $NC^2_B(3/2)$.

We believe that the size of the zero-block of a partition in $NC_B(n)$ should be defined as half of the size of this block seen as a partition in $NC(n)$.  With this definition, the $k$-divisible partitions must belong to $NC_B(kn)$ for some $n$, and for instance,  $NC^2_B(3/2)$ would be empty. We will not continue with this discussion but rather refer the reader to Section $4.5$ of \cite{Arm2009}.
\end{remark}

Of course, our counting results depend on the definition of the size of the zero-block. To avoid this problem we count the number of non-zero block and the number of zero-blocks separately. Also, to avoid confusion we talk about the size of a pair instead of talking about the size of a block.

\begin{definition} Let $\pi\in NC_B(n)$ and $V$ a block of $\pi$. The size of the unordered pair $\{V,-V\}$ is $\frac{1}{2}|V\cup(-V)|$, where for a subset $A\subset \mathbb{Z}$, $|A|$ denotes its cardinality.
\end{definition}

We proceed as for the case of classical non-crossing partitions. We start by calculating the expected number of non-zero blocks of a given size $t$, subject to the restriction of having $m$ non-zero blocks. Before this, we need also to fix the size of the zero-blocks as we state in the following lemma.

\begin{lemma} \label{3B}
The number of pairs of non-zero blocks of size $t$ of all non-crossing partitions in $NC(B_n)$ with $m$ pairs of non-zero blocks and a zero-block of size $2s$ is given by
\begin{equation}n\binom{n-1}{m-1}\binom{n-t-s-1}{m-2}. \label{pertypeB}
\end{equation}
\end{lemma} 
\begin{proof}
 We use  (\ref{rtype partitionsB}) and Lemma \ref{lemma1} in the same way as for the classical non-crossing partitions.
\begin{eqnarray*}
\sum_{_{\substack{ r_{1}+r_{2}+\cdots r_{n}=m  \\ s+r_{1}+2r_{2}+\cdots nr_{n}=n}}}
\frac{n!r_t}{(n-m)!p_r}
&=&n\binom{n-1}{m-1}\sum_{\substack{ r_{1}+r_{2}+\cdots r_{n}=m  \\ %
r_{1}+2r_{2}+\cdots nr_{n}=n-s}}\frac{(m-1)!r_t}{p_r}\\
&=&n\binom{n-1}{m-1}\binom{n-t-s-1}{m-2}.
\end{eqnarray*}%
\end{proof}

\begin{proposition}\label{P1}
The sum of non-zero pairs of blocks$\{-V,V\}$ of size $t$ over all the non-crossing partitions in $NC_B(n)$ is given by 
\begin{equation}
n\binom{2n-t-1}{n-1}.
\end{equation}
\end{proposition}

\begin{proof}
First, we use Lemma \ref{3B} and sum over all possible sizes $s$ of the zero-block in (\ref{pertypeB}). Thus, the number of non-zero blocks of size $t$ of all the non-crossing partitions in $NC_B(n)$ with $m$ pairs of non-zero blocks is given by
\begin{equation}
n\binom{n-1}{m-1}\binom{n-t}{m-1}. \label{pertypeB2}
\end{equation}
Next, we sum over $m$ in \ref{pertypeB2}, using  Chu-Vandermonde's identity for $x=n-1$, $y=n-t$ and $s=n-1$. 
\end{proof}

 Now, we count the number of zero-blocks of size $2t$ over \emph{all} the partitions. 
\begin{proposition}\label{P2}
The number of  non-crossing partitions in $NC_B(n)$ with a zero-block of size $2t$ is given by
\begin{equation}
\binom{2n-t-1}{n-1}.
\end{equation}
\end{proposition}
\begin{proof}

 For counting the number of zero-blocks of a certain size we first we calculate the number of partitions with a zero-block of size $2t$ and $m$ non-zero blocks using \ref{rtype partitionsB} and Lemma \ref{lemma1}, yielding
$$\binom{n}{m}\binom{n-t-1}{m-1}.$$
Summing over $m$ by means of Chu-Vandermonde's identity we get the desired result.
\end{proof}

It is straightforward to pass to $k$-divisible partitions, and obtain Theorem 2.

\begin{mmtheorem}\label{P3}
1) The sum of  the non-zero pairs of blocks $\{-V,V\}$ of size $tk$ of over \emph{all}  non-crossing partitions in $NC^k_B(n)$ is given 
\begin{equation}
nk\binom{n(k+1)-t-1}{nk-1}.
\end{equation}
2) The number of the non-crossing partitions in $NC^k_B(n)$ with a zero-block of size $2t$ then 
\begin{equation}
\binom{n(k+1)-t-1}{nk-1}.
\end{equation}
\end{mmtheorem}

\begin{corollary}
The expected number of unordered pairs $\{-V,V\}$ of size $tk$ of a non-crossing partition chosen uniformly at random in $NC^k_B(n)$ is given by
\begin{equation} \label{expectedtypeB} \frac{(nk+1)\binom{n(k+1)-t-1}{nk-1}}{\binom{(k+1)n}{n}}.
\end{equation}
\end{corollary}

The fact that formulas (\ref{expectedtypeA}) and (\ref{expectedtypeB}) coincide may be explained by the remarkable observation by Biane, Goodman and Nica \cite{BGN} that there is an $(n+1)-to-1$ map from  $NC_B(n)$ to $NC(n)$. More precisely, for any subset $X\subset\mathbb{Z}$, let $Abs(X):=\{|x|:x\in X\}$ and for a partition $\pi\in NC(n)$, denote by $Abs(\pi)$ the set $\{Abs(V):V\in\pi\}$.
\begin{proposition}
Let $n$ be a positive integer. Then the map $
Abs:NC_B(n) \rightarrow NC(n)$\
is a $(n+1)-to-1$ map from $NC_B(n)$ onto $NC(n)$. 
\end{proposition}

Finally, we consider the asymptotics when $n\rightarrow\infty$. Since when $n$ is large the number of zero-blocks is negligible, independent of the choice of how to count them, we obtain the same behavior as for the classical partitions.

\begin{corollary}
When $n\rightarrow\infty$ the expected number of pairs of blocks of size $t$ of a non-crossing partition chosen uniformly at random in $NC^k(B_n)$ is asymptotically $\frac{nk}{(k+1)^{t+1}}.$
\end{corollary}

\section{Number of blocks in k-divisible partitions of type D}\label{S4}

Krattenthaler and M\"uller \cite{KraMu} solved the problem of finding a combinatorial realization of $k$-divisible non-crossing partitions of type $D$.

Let $\pi$ be a partitions of  $\{1,2,...kn,-1,-2,...,-kn\}$. We represent $\pi$ in an annulus with the integers $1,2,...,k(n-1),-1,...,-(kn-k)$ on the outer circle in clockwise order and the integers $k(n-1)+1,...kn,-(k(n-1)+1),...,-kn$ in the inner circle in counterclockwise order. A block $V=\{e_1,e_2,...,e_s\}$ with $e_1,e_2,...e_i$ are arranged on the outer circle in the clockwise order  and $e_{s+1},e_{s+2},...,e_t$ arranged on the inner circle in counterclockwise order for some $i$. Then we draw curves, which lie inside of the annulus, connecting $e_i$ and $e_{i+1}$ for all $i=1,2,..t$, where $e_{t+1} =e_1$. If we can draw the
curves for all the blocks of $\pi$ such that they do not intersect, then we call $\pi$ non-crossing partition on the $(2k(n-1),2k)$ annulus.  We say that a block on the outer circle is visible if there is no block between this block and the inner circle.

Following Kim \cite{Kim} the combinatorial realization of $k$-divisible non-crossing partitions of type $D$ is defined as follows.

\begin{definition}
A $k$-divisible non-crossing partition $\pi$ of type $D_n$ is a non-crossing partition $\pi$ on the $(2k(n-1), 2k)$ annulus satisfying the following conditions:
\begin{enumerate}
 \item 
If $V\in\pi$, then $-V\in\pi$.
 \item
For each block $V\in\pi$ with $t$ elements, if $e_1,e_2,...,e_t$ are the elements of $V$ such that
$e_1,e_2, . . . ,e_s$ are arranged on the outer circle in clockwise order and $e_{s+1},e_{s+2},... ,e_t$ are arranged
on the inner circle in counterclockwise order, then $|e_{i+1}|\equiv|e_i| + 1 $ mod $k$ for all $i=1,2,..t$,
where $e_{t+1} =e_1$.
 \item
If there is a zero-block $V$ of $\pi$, then $V$ contains all the integers on the inner circle and at least two integers on the outer circle.
 \item
If  $\pi$ has no block with elements in both the inner circle and the outer circle, then for any inner block $V_i$ and any visible block $V_j$ and the partition  $\tilde \pi$  with $V_i\cup V_j$ instead of $V_i$ and $V_j$ satisfies Condition 2.

 We denote by $NC^k_D(n)$ the set of $k$-divisible non-crossing partitions of type $D_{n}$
\end{enumerate}

\end{definition}

\begin{remark} As pointed out in \cite{Kim}, Condition 4 was overlooked in \cite{KraMu}. Also, Condition 4 is stated differently in \cite{Kim}. However, for our purposes they are equivalent.
\end{remark}

The number of $k$-divisible non-crossing partitions of type $D$ was given in \cite{Arm2009}.

\begin{proposition}\label{NCD}The number of $k$-divisible non-crossing partitions of type $D_n$ is given by 
$$\binom{(k+1)(n-1)}{n}+\binom{(k+1)(n-1)+1}{n}.$$
\end{proposition}

We consider two cases which we call type $D1$ and type $D2$ depending on whether there is a zero-block or not.
\begin{figure}[h]
\begin{center}
\epsfig{file=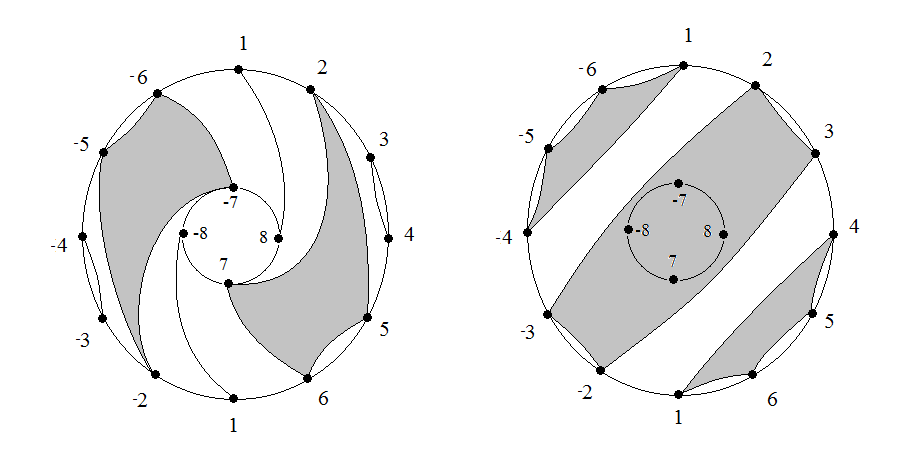, width=10cm}
\caption{2-divisible partitions of type $D1$ and type $D2$}
\end{center}
\end{figure}

\subsection*{Type D1}

\begin{definition}
A $k$-divisible non-crossing partition of type $D_n$ is of type $D1_n$ if it has a zero-block We denote by $NC^k_{D1}(n)$ the set of $k$-divisible non-crossing partitions of type $D1_{n}$. 
\end{definition}

The number of $k$-divisible non-crossing partitions of type $D1$ was calculated in \cite{KraMu}.
\begin{proposition}
\label{NC type D1} Let $n$ be a positive integer, $s\geq2$ and $r_{1},r_{2},...r_{n}\in \mathbb{N}
\cup \{0\}$ be such that $s+r_{1}+2r_{2}...+nr_{n}=n$. Then the number of
partitions of $\pi $ in $NC^k_{D1}(n)$ with $r_{i}$ pairs $\{-V,V\}$ of non-zero blocks size $i$ a zero-block of size $2s$ is given by
\begin{equation}
\frac{(k(n-1))!}{(k(n-1)-m)!p_r} \label{rtype partitionsD1}.
\end{equation}
\end{proposition}

\begin{lemma} \label{3D}
The number of pairs of non-zero blocks of size $tk$ of all non-crossing partitions in $NC^k_{D1}(n)$ with $2m$ non-zero blocks and a zero-block of size $2s$ is
\begin{equation}
(k(n-1))\binom{k(n-1)-1}{m-1}\binom{n-t-s-1}{m-2}. \label{pertypeD1}
\end{equation}
\end{lemma} 
\begin{proof}
In the same way as for the non-crossing partitions of type $B$, we use (\ref{rtype partitionsD1}) and  Lemma \ref{lemma1}.
\begin{eqnarray*}
\sum _{\substack{ r_{1}+r_{2}+\cdots r_{n}=m  \\ r_{1}+2r_{2}+\cdots nr_{n}=n-s}}%
\frac{(k(n-1))!r_t}{(n-m)!p_r} 
&=&(k(n-1))\binom{k(n-1)-1}{m-1}\sum_{\substack{ r_{1}+r_{2}+\cdots r_{n}=m  \\ %
r_{1}+2r_{2}+\cdots nr_{n}=n-s}}\frac{(m-1)!r_t}{p_r}\\
&=&n\binom{k(n-1)-1}{m-1}\binom{n-t-s-1}{m-2}.
\end{eqnarray*}
\end{proof}

\begin{proposition}\label{P1D}
The number of non-zero blocks of size $tk$ of all the non-crossing partitions in $NC^k_{D1}(n)$ then 
\begin{equation}
k(n-1)\binom{(k+1)(n-1)-t-2}{k(n-1)-1}.
\end{equation}
\end{proposition}
\begin{proof}
We use Lemma \ref{3D} and sum over all possible sizes $s$ of the zero-block in (\ref{pertypeD1}). The only difference with type $B$ is that now $s\geq2$. Thus, the number of non-zero blocks of size $tk$ of all the non-crossing partitions in $NC^k_{D1}(n)$ with $m$ non-zero block is given by
\begin{equation*}
k(n-1)\binom{k(n-1)-1}{m-1}\binom{n-t-2}{m-1}. 
\end{equation*}
Summing over $m$, using  Chu-Vandermonde's identity for $x=k(n-1)-1$, $y=n-t-2$ and $s=k(n-1)-1$, we get the result.
\end{proof}

In a similar way as for type $B$, we calculate the number of zero-blocks of a given size over all the non-crossing partitions of type $D$.
\begin{proposition}\label{P0D}
The number of non-crossing partitions in $NC_D^k(n)$ with a zero-block of size $2tk$ is given by 
\begin{equation}
\binom{(k+1)(n-1)-t}{k(n-1)-1}.
\end{equation}
\end{proposition}
\begin{proof}
This is analogous to the  case of $NC^k_B(n)$. 
\end{proof}

\subsection*{Type D2}

\begin{definition}
A $k$-divisible non-crossing partition of type $D_n$ is of type $D2_n$ if it has no zero-block. We denote by $NC^k_{D2}(n)$ the set of $k$-divisible non-crossing partitions of type $D2_{n}$. 
\end{definition}

The number of $k$-divisible non-crossing partitions of type $D2$ was calculated in \cite{KraMu}.
\begin{proposition}
\label{NC type D2} Let $n$ be a positive intege, $r_{1},r_{2},...r_{n}\in \mathbb{N}
\cup \{0\}$ be such that $r_{1}+2r_{2}...+nr_{n}=n$. Then the number of
partitions of $\pi $ in $NC^k(D2)$ with $r_{i}$ non-zero pairs of size $ik$ block is given by
\begin{equation}
2\frac{(k(n-1))!}{(k(n-1)-m)!p_r}+r_1\frac{(k(n-1)!}{(k(n-1)-m+1)!p_r}. \label{rtypepartitionsD3}
\end{equation}
where $p_r=r_1!r_2!\cdots r_n!$ and $r_{1}+r_{2}...+r_{n}=m.$
\end{proposition}

\begin{proposition}\label{P2D}
The number of blocks of size $tk$ of all non-crossing partitions in $NC^k_{D2}(n)$ is given by 
\begin{equation}
k(n-1)\left(\binom{(k+1)(n-1)-t-2}{k(n-1)-2}+2\binom{(k+1)(n-1)-t-1}{k(n-1)-2}\right) 
\end{equation}
for $t>1$ and by
\begin{equation}
k(n-1)\left(\binom{(k+1)(n-1)-3}{k(n-1)-2}+2\binom{(k+1)(n-1)-2}{k(n-1)-2}\right)+\binom{(k+1)(n-1)-1}{k(n-1)-1}
\end{equation}
for $t=1$.
\end{proposition}
\begin{proof}
Using (\ref{rtypepartitionsD3}) we count the number of blocks of size $tk$ over all  non-crossing partitions in $NC^k_D(n)$ with $m$ non-zero pairs and then sum over $m$. This is
 \begin{equation*}
2\sum_{_{\substack{r_{1}+\cdots +r_{n}=m  \\ r_{1}+\cdots +nr_{n}=n}}} \frac{r_t(k(n-1))!}{(k(n-1)-m)!p_r}+
\sum_{_{\substack{r_{1}+\cdots +r_{n}=m  \\ r_{1}+\cdots +nr_{n}=n}}}
\frac{r_tr_1(k(n-1))!}{(k(n-1)-m+1)!p_r}.\label{maineqD2}
\end{equation*}
The first part of the sum can be treated in the same way as the classical case using Lemma \ref{lemma1}.
 \begin{eqnarray*}
2\sum_{_{\substack{r_{1}+\cdots +r_{n}=m  \\ r_{1}+\cdots +nr_{n}=n}}} \frac{r_t(k(n-1))!}{p_r(k(n-1)-m)!}
&=&2(k(n-1))\binom{k(n-1)-1}{m-1}\sum_{_{\substack{r_{1}+\cdots +r_{n}=m  \\ r_{1}+\cdots +nr_{n}=n}}} \frac{r_t (m-1)!}{p_r}   \\
&=&2(k(n-1))\binom{k(n-1)-1}{m-1}\binom{n-t-1}{m-2}.
 \end{eqnarray*}
Summing over $m$, by means of the Chu-Vandermonde's formula, we get
\begin{equation} \label{tD2}
2k(n-1)\binom{(k+1)(n-1)-t-1}{k(n-1)-2}.
\end{equation}
For the second part of the sum we divide in two cases, $t=1$ and $t>1$. 

If $t>1$, then making the change of variable $\tilde r_1=r_1-1$, and  $\tilde r_i=r_i$, for $i\neq1$, we get 
\begin{eqnarray*}
\sum_{_{\substack{r_{1}+\cdots +r_{n}=m  \\ r_{1}+\cdots +nr_{n}=n}}}
\frac{r_tr_1(k(n-1))!}{(k(n-1)-m+1)!p_r}
&=&\sum_{_{\substack{\tilde r_{1}+\cdots +\tilde r_{n}=m-1  \\ r_{1}+\cdots +nr_{n}=n-1}}}
\frac{\tilde r_t(k(n-1))!}{(k(n-1)-m+1)!\prod\limits_{i=1}^{n}\tilde r_{i}!}\\
&=&(k(n-1))\binom{k(n-1)-1}{m-2}\binom{n-t-2}{m-3}.
\end{eqnarray*}
Summing over $m$ we get
\begin{equation} \label{tD3}
(k(n-1))\binom{(k+1)(n-1)-t-2}{k(n-1)-2}.
\end{equation}
Thus, combining  (\ref{tD2}) and (\ref{tD3}) we get the result  for $t>1$.

Now, for $t=1$, the sum $$\sum_{_{\substack{r_{1}+\cdots +r_{n}=m  \\ r_{1}+\cdots +nr_{n}=n}}}
\frac{r_1r_1(k(n-1))!}{(k(n-1)-m+1)!p_r}$$
can be split as 
$$\sum_{_{\substack{r_{1}+\cdots +r_{n}=m  \\ r_{1}+\cdots +nr_{n}=n}}}
\frac{r_1(r_1-1)(k(n-1))!}{(k(n-1)-m+1)!p_r}+\sum_{_{\substack{r_{1}+\cdots +r_{n}=m  \\ r_{1}+\cdots +nr_{n}=n}}}
\frac{r_1(k(n-1))!}{(k(n-1)-m+1)!p_r}.$$
Each of the terms can be treated as before, yielding
\begin{equation*}
 k(n-1)\binom{k(n-1)-1}{m-2}\binom{n-3}{m-3}+\binom{k(n-1)}{m-1}\binom{n-2}{m-2}.
\end{equation*}
Again, summing over $m$ we get 
\begin{equation}\label{tD4}
k(n-1)\binom{(k+1)(n-1)-3}{k(n-1)-2}+\binom{(k+1)(n-1)-1}{k(n-1)-2}.
\end{equation}
Combining  (\ref{tD2}) and (\ref{tD4}) we get the result for $t=1$.
\end{proof}

Now putting together Propositions \ref{P0D}, \ref{P1D} and \ref{P2D} we get, after some small simplifications, Theorem $3$.

\begin{mmtheorem}
The sum of the number of pairs of blocks $\{V,-V\}$ of size $tk$ over all the partitions in $NC^k_D(n)$ is
$$(k(n-1)+1)
\binom{(k+1)(n-1)-t}{k(n-1)-1}+k(n-1)\binom{(k+1)(n-1)-t-1}{k(n-1)-2}.$$
\end{mmtheorem}

Notice that even though the formulas of Proposition \ref{P2D} look very different, for $t=1$ and for $t>1$, Theorem \ref{T3} gives the same formula for all $t$.

Finally, using the Theorem 3 and Proposition \ref{NCD} we get the asymptotics when $n\rightarrow\infty$.  Surprisingly, we get the same behavior for the cases of $NC^k(n)$ and $NC^k_B(n)$.

\begin{corollary}
When $n\rightarrow\infty$ the expected number of pairs of size $tk$ of a non-crossing partition chosen uniformly at random in $NC^k(D_n)$ is asymptotically $\frac{nk}{(k+1)^{t+1}}$.
\end{corollary}

\section{The bijection}\label{S5}

In this section we give a bijective proof of the fact that $NC^k(n)=NC_{k+1}(n)$. From this bijection we derive Corollary \ref{expected}.

\begin{lemma}
\label{bij}
For each $n$ and each $k$ let $f:NC_{k+1}(n)\rightarrow NC^{k}(n)$ be the
map induced by the identification of the pairs $\{k+1,k+2\},
\{2(k+1),2(k+1)+1\},\dots,\{n(k+1),1\}.$ Then $f$ is a bijection.

\end{lemma}
\begin{proof}
First, we see that the image of this map is in $NC^{k}(n).$ So, let $\pi$ be a $(k+1)$-equal partition.

(i) Every block has one element on each congruence $mod$ $k+1$. Indeed,
because of the characterization of non-crossing partitions on Remark \ref{4},
there is at least one interval, which has of course this property. Removing
this interval does not affect the congruence in the elements of other blocks. So by induction on $
n$ every block has one element of each congruence $mod$ $k+1$.

(ii) Note that for each two elements identified we reduce 1 point. So suppose
that $m$ blocks (of size $k$) are identified in this bijection to form a big
block $V$. Then the number of vertices in this big block equals $m(k+1)-\#($
identified vertices$)/2$. Now, by (i), there are exactly two elements in each
block to be identified with another element, that is $2m$ . So 
\begin{eqnarray*}
\left\vert V\right\vert &=&m(k+1)-\#(\text{identified vertices})/2 \\
&=&m(k+1)-(2m)/2=mk.
\end{eqnarray*}%
this proves that $f(\pi )\in NC^{k}(n).$

Now, it is not very hard to see that by splitting the points of $\pi \in NC^{k}(n)$ we get a unique inverse $f^{-1}(\pi )\in $ $NC_{k+1}(n).$
\end{proof}

\begin{figure}[here]
\begin{center}
\epsfig{file=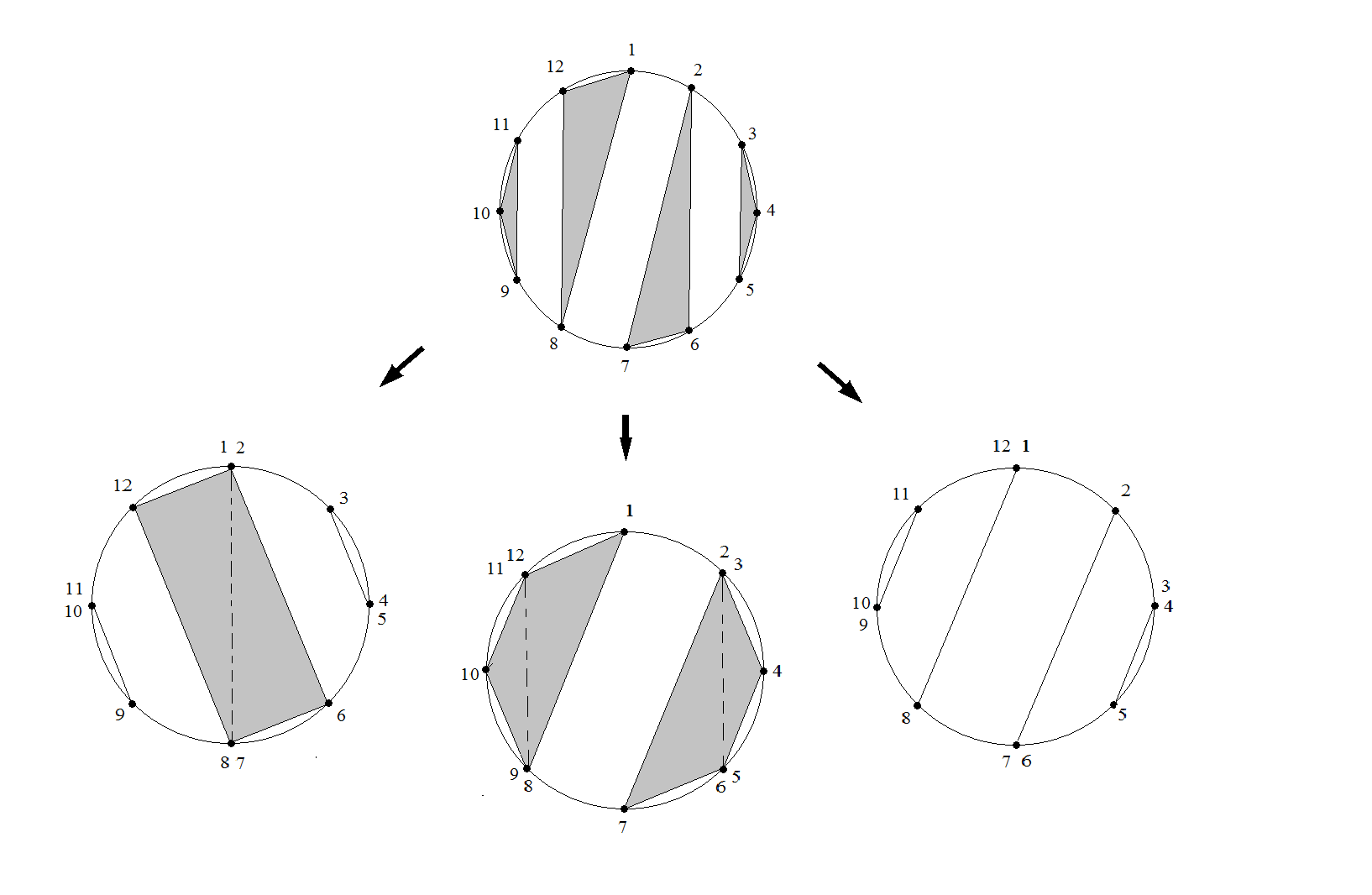, width=12cm}
\caption{Bijections between 3-equal and 2-divisible non-crossing partitions}
\end{center}
\end{figure}

Now we give a proof of Corollary \ref{expected}.

\begin{proof}

For each $n$ and each $k$  and each $0<i\leq k+1$ let $f_i:NC_{k+1}(n)\rightarrow NC^{k}(n)$ the
map induced by the identification of the pairs $\{k+1+i,k+1+i+1\},\dots \{2(k+1)+i,2(k+1)+i+1\},...\{n(k+1)+i,n(k+1)+i+1\}$  (we consider elements $mod$ $nk$). 
Then  by the proof of the previous lemma, each $f_i$ is a bijection.
 So, let $\pi$ be a fixed $(k+1)$-equal partition. Considering all the bijections $f_i$ on this fixed partition, we see that every point $j$ is  identified twice (one with  $f_{j-1}$ and one with $f_j$). So for each partition $\pi$ in $NC_{k+1}(n)$, the collection $(f_i(\pi))^k_{i=1}$  consists of  $k+1$ partitions in $NC^k(n)$ whose number of blocks add $kn+1$.
\end{proof}

\begin{remark}
Note that for a k-divisible partition on $[2kn]$ points, the property of being centrally symmetric is preserved under the bijections $f_i$ (see e. g. Fig.5), and then the arguments given here also work for the partitions of type $B$.
 We expect that a similar argument works for type $D$. 
\end{remark}  

 In the following example we want to illustrate how the bijection given by Lemma \ref{bij} allows us to count $k$-divisible partitions with some restrictions by counting the preimage under $f$.

\begin{example} 
Let $NC_{1\rightarrow 2}^k(n)$ be the set of $k$-divisible non-crossing partitions of $[kn]$ such that $1$ and $2$ are in the same block. It is clear that  $\pi\in NC_{1\rightarrow 2}^k(n)$ if and only if  $f^{-1}(\pi)$  satisfies that $1$ and $2$ are in the same block.

Now, counting the $(k+1)$-equal non-crossing partitions of $[(k+1)n]$ such that $1$ and $2$ are in the same blocks is the same as counting non-crossing partitions of $[(k+1)n-1]$ with $n-1$ blocks of size $k+1$ and $1$ block of size $k$ containing the element $1$, since $1$ and $2$ can be identified. From Proposition \ref{NC type}, the size of this set is easily seen to be 
$$\frac{k}{(k+1)n-1}\binom{(k+1)n-1}{n-1}=\frac{k}{n-1}\binom{(k+1)n-2}{n-2}$$ 
where the first factor  of the LHS is the probability that the block of size $k$ contains the element $1$.
\end{example}

Let us finally mention that the bijections $f_i$ are closely related to the Kreweras complement of a $(k+1)$-equal non-crossing partitions, which was considered in \cite{ArVar}. Indeed $Kr(\pi)$ can be divided in a canonical way into $k+1$ partitions of [n],  $\pi_1,...,\pi_{k+1}$, such that  $\left\vert \pi_i \right\vert = \left\vert f_i(\pi) \right\vert $. Fig. 5 shows the bijections $f_1,f_2$ and $f_3$ for $k=3$, $n=4$ and $\pi=\{\{1,8,9\},\{2,6,7\},\{3,4,5\},\{9,10,11\}\}$, while Fig. 6  shows the same partition as Fig. 5 with its Kreweras complement divided into the partitions $\pi_1,\pi_2$ and $\pi_3$.

\begin{figure}[here]
\begin{center}
\epsfig{file=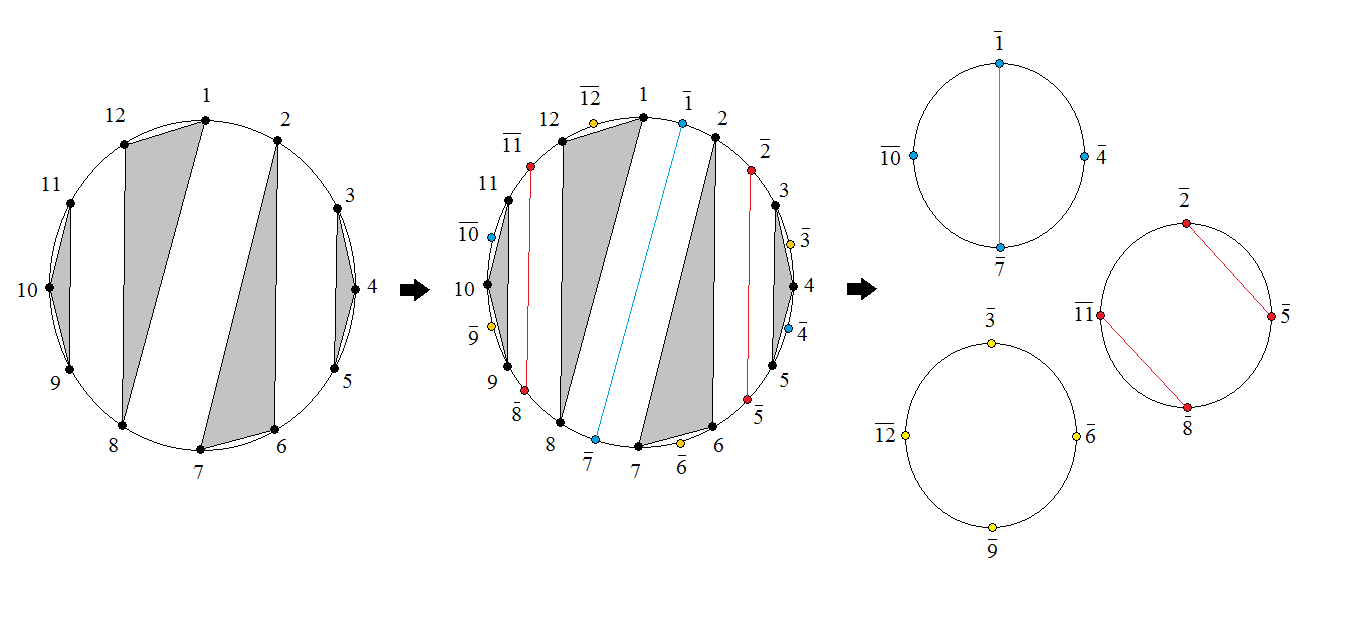, width=12cm}
\caption{A 3-equal and its Kreweras complement divided mod 3.}
\end{center}
\end{figure}

\section*{Acknowledgements}

The author is grateful to Janosch Ortmann for asking some of the questions treated in this paper and would like to thank Pablo Sober\'{o}n for comments that helped improving this paper. He is indebted to Professor Christian Krattenthaler for bringing to his attention non-crossing partitions of type $B$ and $D$ and for making him available the preprint \cite{Kra}.

\end{document}